\documentclass[reqno,11pt,twoside,english]{amsart}
\usepackage[T1]{fontenc}
\usepackage[latin9]{inputenc}
\usepackage{amsthm}
\usepackage{amssymb}
\usepackage{amsmath,amsfonts,epsfig}

\makeatletter

\numberwithin{equation}{section}

\usepackage{color}
\usepackage[final,linkcolor = blue,citecolor = blue,colorlinks=true]{hyperref}

\usepackage{type1cm,ae}
\usepackage{graphicx}
\usepackage{xspace}
\usepackage{setspace}
\usepackage[english]{babel}
\usepackage{tikz}
\usetikzlibrary{arrows,patterns}

\newcommand{\R}{{\mathbb R}}

\newcommand\wq{\infty}

\newcommand\al{\alpha}
\newcommand\be{\beta}

\newcommand\loc{{\mathop\mathrm{\,loc\,}}}

\newcommand\De{\Delta}

\newcommand\ep{\epsilon}

\newcommand\ga{{\gamma}}

\newcommand\Om{{\Omega}}

\newcommand\var{\varphi}

\newcommand\si{\sigma}
\newcommand\pa{\partial}
\newcommand\na{\nabla}

\newcommand\D{\rm d}

\numberwithin{equation}{section}
\newtheorem{theorem}{Theorem}[section]

\newtheorem{corollary}[theorem]{Corollary}

\newtheorem{lemma}[theorem]{Lemma}

\newtheorem{proposition}[theorem]{Proposition}

\theoremstyle{definition}
\begingroup
\newtheorem{definition}[theorem]{Definition}
\newtheorem{remark}[theorem]{Remark}

\endgroup

\def\Xint#1{\mathchoice
	{\XXint\displaystyle\textstyle{#1}}%
	{\XXint\textstyle\scriptstyle{#1}}%
	{\XXint\scriptstyle\scriptscriptstyle{#1}}%
	{\XXint\scriptscriptstyle\scriptscriptstyle{#1}}%
	\!\int}
\def\XXint#1#2#3{{\setbox0=\hbox{$#1{#2#3}{\int}$}
		\vcenter{\hbox{$#2#3$}}\kern-.5\wd0}}

\def\dashint{\Xint-}

\makeatother

\usepackage{babel}

\begin{document}

\title[Regularity of $P$-harmonic mappings]{Some Regularity results for $P$-harmonic mappings between Riemannian manifolds}

\author{Chang-Yu Guo and Chang-Lin Xiang*}

\address[Chang-Yu Guo]{Institute of Mathematics, \'Ecole Polytechnique F\'ed\'erale de Lausanne (EPFL), Section 8,  CH-1015 Lausanne, Switzerland and School of Mathematics, Shandong University, P. R. China}
\email{guocybnu@gmail.com}

\address[Chang-Lin Xiang]{School of Information and Mathematics, Yangtze University, Jingzhou 434023, P.R. China}
\email{changlin.xiang@yangtzeu.edu.cn}

\thanks{* Corresponding author}
\thanks{C.-Y. Guo was supported by Swiss National Science Foundation Grant 165848, 165507, 175985 and the Qilu funding of Shandong University. The corresponding author C.-L. Xiang is financially supported by the National Natural Science Foundation of China (No. 11701045) and by  the Yangtze Youth Fund (No. 2016cqn56).}

\begin{abstract}
Let $M$ be a $C^2$-smooth Riemannian manifold with boundary and $N$ a complete $C^2$-smooth Riemannian manifold. We show that each stationary $p$-harmonic mapping $u\colon M\to N$, whose image lies in a compact subset of $N$, is locally $C^{1,\alpha}$ for some $\alpha\in (0,1)$, provided that $N$ is simply connected and has non-positive sectional curvature. We also prove similar results for minimizing $p$-harmonic mappings with image being contained in a regular geodesic ball. Moreover, when $M$ has non-negative Ricci curvature and $N$ is simply connected with non-positive sectional curvature, we deduce a  gradient estimate for  $C^1$-smooth weakly $p$-harmonic mappings from which follows a Liouville-type theorem in the same setting.
\end{abstract}

\maketitle

{\small
\smallskip
\newline
}
\bigskip

\section{Introduction and main results}
Let $(M,g)$ and $(N,h)$ be two Riemannian manifolds with $\dim M=n$ and let $1<p<\infty$ a constant. A $p$-harmonic mapping $u:M\to N$ is a  critical point of the energy functional $\int_M |\na u|^p\D \mu.$ Regularity theory for $p$-harmonic mappings between Riemannian manifolds have been explored extensively in the literature, see subsection \ref{subsec: bg} below for more details. In this note, our aim is to enrich some regularity results in this respect, particularly in the case $1<p<2$.

\subsection{Background}\label{subsec: bg}

The research on harmonic mappings (i.e. $p=2$) has a long and distinguished history, making it one of the most central topics in geometric analysis on manifolds~\cite{sy97}. Since it is almost impossible to describe all the relevant works, we only briefly introduce some important works which are largely related to  our problem. In his pioneering work~\cite{m48}, Morrey proved the H\"older continuity of {minimizing} harmonic mappings when $n=2$ (and smooth if $M$ and $N$ are smooth).  The breakthrough of higher dimensional theory for harmonic mappings was made by Eells and Sampson~\cite{es64}, where they proved that every homotopy class of mappings from a closed manifold $M$ into $N$ has a smooth harmonic representative, if $N$ has non-positive (sectional) curvature. Important progress were made later by Hartman~\cite{h67} and Hamilton~\cite{h75}. When the image of a (weakly) harmonic mapping $u$ is contained in a regular geodesic ball of $N$, the existence, uniqueness and regularity theory were substantially developed by Hildebrandt and Widman~\cite{hw77},  {J\"ager and Kaul \cite{Jager-Kaul-1979}} and Hildebrandt, Kaul and Widman~\cite{hkw77}. In particular, it was proved in~\cite{hkw77} that each (weakly) harmonic mapping $u\colon M\to N$ is smooth whenever $u(M)$ is contained in a regular geodesic ball $B_R(P)$ of $N$ (see Definition~\ref{def:regular ball} below for the precise definition of regular geodesic ball). This result is optimal in the sense that the result fails if we enlarge the radius $R$ of the geodesic ball $B_R(P)$ (so that $B_R(P)$ fails to be regular). In the Euclidean setting, important results were obtained by Giaquinta and Giusti~\cite{gg82} for the case { where} the image of a (locally minimizing) harmonic mappings lie in a coordinate chart. The regularity theory for (minimizing) harmonic mappings into general target Riemannian manifolds was later developed by Schoen and Uhlenbeck in their seminal paper~\cite{su82} (see also~\cite{su83} for boundary regularity theory and~\cite{su84} for the case $N=\mathbb{S}^n$). In particular, Schoen and Uhlenbeck proved that minimizing harmonic mappings are smooth away from a small singular set with Hausdorff dimension no more than $n-3$. { Later, Lin \cite{l99} provided a necessary and sufficient condition for gradient estimates of stationary harmonic mappings. In particular, he showed that if the universal cover of $N$ supports a pointwise convex function, then every smooth stationary harmonic mapping enjoys a global gradient estimates under suitable assumptions on the boundary $\pa M$. He also showed that the singular set of stationary harmonic mappings has dimension less than or equal to $n-4$, under the assumption that $N$ has no smooth nonconstant harmonic sphere $\mathbb{S}^2$.}  {T}he structure of singular sets (of minimizing and stationary harmonic mappings) has gained deeper understanding in the recent works~\cite{cn13,l99,nv17}; see also~\cite{rs08} for an elegant new approach  for the regularity result of weakly harmonic mappings.

General $p$-harmonic mappings, $1<p<\infty$, also gained growing interest in the past decades; see for instance~\cite{df90,fr02,fr03,f88,f89,f90,ff88,gw17,hl87,hlw97,Luckhaus88,nvv16}.
Relying on the fundamental work of Struwe~\cite{s88}, Fardoun and Regbaoui~\cite{fr02,fr03} developed the theory of $p$-harmonic mapping flow and partially extended the results of Eells and Sampson~\cite{es64} to $p$-harmonic mappings. Concerning the (partial) regularity result for general Riemannian targets, Hardt and Lin~\cite{hl87}, Luckhaus \cite{Luckhaus88}, and Fuchs~\cite{f90} have extended the regularity result of Schoen and Uhlenbeck~\cite{su82} to minimizing $p$-harmonic mappings ($1<p<\infty$). More precisely, they proved that minimizing $p$-harmonic mappings (between compact smooth Riemannian manifolds) are locally $C^{1,\alpha}$ away from a singular set with Hausdorff dimension at most $n-[p]-1$, where the singular set is defined as
\begin{equation}\label{eq:def for singular set}
	S_u:=\Big\{a\in M: \limsup_{r\to 0} r^{p-n}\int_{B_r(a)}|\nabla u|^p\D \mu>0\Big\}.
\end{equation}
The structure of singular set has gained deeper understanding more recently in~\cite{hlw97,cn13,nvv16}. { As to weakly $p$-harmonic mappings, we would like to mention the interesting work of Fardoun and Regbaoui \cite{fr13}, where the authors  found a small constant $\ep_0$ such that  if $u:\Om\subset M\to N$ is a weakly $p$-harmonic mapping with $u(\Om)$ contained in a regular geodesic ball of radius $\ep_0$, then $u\in C^{1,\al}(\Om,N)$ for some $0<\al<1$.  This result  partially generalized the result of Hildbrandt et al. \cite{hkw77}. They also proved a uniqueness result for weakly $p$-harmonic mappings.}

In the spirit of Schoen and Uhlenbeck~\cite{su82}, Hardt and Lin~\cite{hl87}, Luckhaus~\cite{Luckhaus88} and Fuchs~\cite{f90} etc., it is natural to find geometric restrictions that exclude the singular set $S_u$ of a minimizing $p$-harmonic mapping $u\colon M\to N$. That is, we look for geometric conditions to ensure that each minimizing $p$-harmonic mapping is regular everywhere on $M$. In \cite[Theorem IV]{su82} and \cite[Theorem 4.5]{hl87}, the authors have developed some criteria to exclude the singular set for (minimizing) harmonic and $p$-harmonic mappings. As a corollary of their main results, Schoen and Uhlenbeck \cite[Corollary]{su82} proved that if either the target manifold $N$ has non-positive curvature or the image of a minimizing harmonic mapping lies in a strict convex ball in $N$, then the harmonic mapping is smooth. This is closely related to the earlier work of Eells and Sampson~\cite{es64} and  Hildebrandt, Kaul and Widman~\cite{hkw77}. In \cite[Theorem 4.5]{hl87}, it was proved that \emph{if each $p$-minimizing tangent mapping  from the unit ball in $\R^l$ into $N$ is constant for  $l=1,2,\dots,n$, then $S_u=\emptyset$ for each minimizing $p$-harmonic mapping $u\colon M\to N$}.

On the other hand, if we impose certain geometric restrictions on the manifold $N$ or on the image of $M$ under $u$, then some partial results for $S_u=\emptyset$ are well-known. In particular, when the image of $M$ of a {minimizing} $p$-harmonic mapping $u$ is contained in a regular geodesic ball in $N$, {the previous criteria of Hardt and Lin, together with the Liouville theorem proved by Fuchs~\cite{f89}, implies} that $S_u=\emptyset$ for each { minimizing} $p$-harmonic mapping $u\colon M\to N$ with $p\geq 2$; { see also related result by Fuchs \cite{f89} for stationary $p$-harmonic mappings}. If $N$ is simply connected and has non-positive sectional curvature, Wei and Yau~\cite{wy94} proved that each $p$-minimizing tangent mapping of $u$ from the unit ball in $\R^l$ into $N$ is constant for each $l=1,2,\dots,n$,   whenever it enjoys certain a priori regularity for $p\geq 2$ and so $S_u=\emptyset$ in this case by the criteria of Hardt and Lin.

In view of the above-mentioned works, two interesting and basic questions regarding the regularity theory of $p$-harmonic mapping between Riemannian manifolds can be formulated as follows:

\begin{description}
	\item[Regularity Question (NPC)]  \emph{Are $p$-harmonic mappings $u\colon M\to N$, $1<p<\infty$, necessarily locally $C^{1,\alpha}$ if $N$ is simply connected and has non-positive sectional curvature?}
	\smallskip
	
\item[Regularity Question (Regular ball)] \emph{Are $p$-harmonic mappings $u\colon M\to N$, $1<p<\infty$, necessarily locally $C^{1,\alpha}$ if $u(M)$ is contained in a regular geodesic ball $B_R(P)\subset N$?}
\end{description}

In the present note, we shall provide (partial) affirmative answers to the above two questions. Before stating our main results, let us point out some difficulties that will occur and our strategies and innovations. For the first problem, that is,  when $N$ has non-positive curvature, the regularity method of Hardt and Lin~\cite{hl87} (and also~\cite{Luckhaus88,f90}) necessarily generates singular sets for minimizing $p$-harmonic mappings $u\colon M\to N$, and the criteria mentioned above (to deduce that the singular set $S_u$ is empty) seems not { to be} working directly without any further a priori regularity assumption for $u$. The argument of Schoen and Uhlenbeck ~\cite[Corollary]{su82} also fails in our setting as composition of (square of) the distance function with a $p$-harmonic mapping fails in general to be a sub-$p$-harmonic function. To overcome these difficulties, we will combine some ideas from Gromov-Schoen \cite{gs92}. 
For the second problem,
we revisit the famous paper of Hildebrandt, Kaul and Widman \cite{hkw77} and apply some delicate estimates on curvatures to derive an important Caccioppoli type inequality, from which follows a Liouville type theorem for $p$-harmonic mappings from $\R^l$ to the regular geodesic ball. Then, the criteria of Hardt-Lin \cite{hl87} for singular set applies.

After answering the above two problems, we shall {furthermore} derive some estimates on gradient of $C^1$-smooth weakly harmonic mappings. These estimates will {lead} to certain Liouville type theorem for $p$-harmonic mappings on complete non-compact Riemannian manifolds with nonnegative curvature.

\subsection{Main results}\label{subsec:main results}
The setting of our problems is as follows. Let $M$ be an $n$-dimensional $C^2$-smooth Riemannian manifold with boundary $\partial M$ and $N$ a complete $C^2$-smooth Riemannian manifold. For simplicity, we assume that $N=(N,h)$ is isometrically embedded into some Euclidean space $\R^k$. Throughout this paper, we assume that $p\in (1,\infty)$.

Fix a domain $\Omega\subset M$. The Sobolev space $W^{1,p}(\Omega,N)$, $1<p<\infty$, is defined as
\begin{equation*}\label{eq:def for Sobolev}
	W^{1,p}(\Omega,N):=\Big\{u\in W^{1,p}(\Omega,\R^k): u(x)\in N \text{ for a.e.  } x\in \Omega \Big\},
\end{equation*}
where $W^{1,p}(\Omega,\R^k)$ is the usual $\R^k$-valued Sobolev space. For $u,v\in W^{1,1}(\Omega,\R^k)$, the inner product $\langle \nabla u,\nabla v\rangle$ is well-defined for almost every point on $\Omega$ by
$$\langle \nabla u,\nabla v\rangle=\sum_{\alpha,\beta}g^{\alpha\beta} \frac{\partial u}{\partial x^\alpha}\cdot \frac{\partial v}{\partial x^\beta}
$$
where $g^{\alpha\beta}=[g_{\alpha\beta}]^{-1}$ is the inverse of the matrix representing the metric $g$ of $M$ in local coordinates $x_1,\cdots,x_n$. The energy density for $u\in W^{1,p}(\Omega,N)$ is defined as
$$e_p(u)=|\nabla u|^p:=\langle \nabla u,\nabla u\rangle^{p/2}$$
and the $p$-energy of $u$ is given by
$$E_p(u)=\int_{\Omega}|\nabla u|^pd\mu.$$

A mapping $u\in W^{1,p}(\Omega,N)$ is said to be \emph{weakly $p$-harmonic} if it is a critical point of $E_p(u)$ with respect to variations in the target manifold $N$.  In particular, for any compactly supported vector field $\psi\in W^{1,p}_0(M,\R^k)\cap L^\infty(M,\R^k)$ with $\psi (x)\in T_{u(x)}N\subset\R^k$ for a.e. $x\in \Om$, we let $u_t=\exp_{u(x)}\big(t\psi(x) \big)$. Then, it holds
\begin{equation}\label{eq:variation for weakly p-harmonic}
\frac{d}{dt}\Big|_{t=0}E_p(u_t)=\int_{\Omega}\langle |\nabla u|^{p-2}\nabla u, \nabla \psi\rangle d\mu=0.
\end{equation}
 {We mention that another equivalent way to define weakly $p$-harmonic maps applies the nearest point mapping of $N$. Let $\Pi: N_{\delta} \to N$ be the nearest point projection, where $N_{\delta}$ is a small tubular neighborhood of $N$. Then $u$ is a  weakly $p$-harmonic mapping if $\frac{\D}{\D t}E_p(\Pi(u+t\var))=0$ for all $\var\in C_0^{\wq}(\Om, \R^k)$. This leads to the Euler-Lagrange equation satisfied by $u$  }  as follows:
$$-\text{div}\big(|\nabla u|^{p-2}\nabla u\big)=|\nabla u|^{p-2}A(u)(\nabla u,\nabla u),$$
where $A$ is the second fundamental form of $N$ in $\R^k$, see e.g. \cite{fr13}.

If,  in addition, $u$ is a critical point with respect to  variations in the domain, then it is called a \emph{stationary $p$-harmonic mapping}. That is,  a stationary $p$-harmonic mapping $u$  is a weakly $p$-harmonic mapping which also satisfies
$$\frac{d}{dt}\Big|_{t=0}E_p\Big(u\big(\exp_x(t\xi(x)) \big) \Big)=\frac{d}{dt}\Big|_{t=0}\int_{\Omega} \Big|\nabla u\big(\exp_x(t\xi(x)) \big) \Big|^p d\mu=0$$ for every smooth compactly supported vector field $\xi\colon M\to TM$.


Finally, a mapping $u\in W^{1,p}(\Omega,N)$ is called \emph{minimizing $p$-harmonic}, if
$$E_p(u|_{\Omega'})\leq E_p(v|_{\Omega'})$$
for every relatively compact domain $\Omega'\subset \Omega$ and every $v\in W^{1,p}(\Omega,N)$ with the same trace as $u$ on $\partial\Omega'$.
 That is, $u-v\in W^{1,p}_0(\Omega',\mathbb{R}^k)$ holds. { Note that  minimizing $p$-harmonic mappings are automatically  stationary  $p$-harmonic mappings.}

Our first main result reads as follows.
\begin{theorem}\label{thm:main theorem}
Each stationary $p$-harmonic mapping $u\colon \Omega \to N$, whose image lies in a compact subset of $N$, is locally $C^{1,\alpha}$ for some $\alpha\in (0,1)$ if $N$ is simply connected and has non-positive sectional curvature.
\end{theorem}

As commented earlier, Theorem \ref{thm:main theorem} can be viewed as a natural extension of the regularity result of Eells and Sampson~\cite{es64} or~\cite[Corollary]{su82} for harmonic mappings into Riemannian manifolds with non-positive curvature.

To prove Theorem \ref{thm:main theorem}, the main idea is to derive a Morrey type estimate (see Lemma~\ref{lemma:key monotone inequality}) of a $p$-harmonic mapping $u\colon \Omega\to N$. Then it follows immediately that $u$ is locally $C^{0,\alpha}$ and the standard regularity theory (see e.g. Hardt-Lin~\cite[Section 3]{hl87}) gives the desired local $C^{1,\alpha}$ regularity. The approach is inspired by an idea for proving Lipschitz regularity of harmonic mappings into singular metric spaces, due to Gromov and Schoen~\cite{gs92}. More precisely, we follow the idea of Gromov-Schoen~\cite{gs92} to consider the composed function $f_Q:=d^2(u,Q)$ {for a given point $Q\in N$}, and derive {a} certain weak differential inequality (see Lemma~\ref{lemma:equation} below) that relates the $p$-energy of $u$ and the gradient of $f_Q$, which allows us to control the $p$-energy from above by (a constant multiple of) the integration of $|\nabla u|^{p-2}|\nabla f_Q|$ over $\partial B(a,r)$. Then we  use H\"older's inequality and Poincar\'e's inequality to estimate the $p$-energy of $u|_{B(a,r)}$ from above. A crucial technical point here is to use (a Riemannian version of) the monotonocity formula for stationary $p$-harmonic mappings due to Hardt-Lin~\cite{hl87} (see Lemma~\ref{lemma:monotone equation} below).

We next recall the definition of regular geodesic ball from~\cite{hkw77}.
\begin{definition}[Regular geodesic ball]\label{def:regular ball}
	Let $B_R(P)\subset N$ be a geodesic ball centered at $P$ with radius $R$. Let $C(P)$ be the cut locus of its center $P$. We say that $B_R(P)$ is a regular geodesic ball if $B_R(P)\cap C(P)=\emptyset$ and $R<\frac{\pi}{2\sqrt{\kappa}}$, where $\kappa\geq 0$ is an upper bound for the sectional curvature of $N$ on the ball $B_R(P)$ and if it lies within normal range of all of its points.
\end{definition}

Our second main result answers affirmatively the second regularity question, extending~\cite[Theorem {3}]{hkw77} to minimizing $p$-harmonic mappings.
\begin{theorem}\label{thm:main theorem 2}
	Each minimizing $p$-harmonic mapping $u\colon \Omega \to N$, whose image $u(M)$ is contained in a regular geodesic ball  $B_R(P)\subset N$, is locally $C^{1,\alpha}$ for some $\alpha\in (0,1)$.
\end{theorem}

 { Recall that if $N$ is simply connected and has nonpositive sectional curvature, then $N$ is diffeomorphic to the Euclidean space $\R^{{\rm dim} N}$ by the Cartan-Hadamard theorem. Consequently, any ball of finite radius in $N$ is a regular geodesic ball. In this case,  Theorem~\ref{thm:main theorem 2} follows from  Theorem~\ref{thm:main theorem}. }

The main arguments leading to Theorem~\ref{thm:main theorem 2} are due to Hardt-Lin~\cite{hl87} and Fuchs~\cite{f88}. More precisely, in~\cite{f88}, Fuchs has shown that each $p$-harmonic mapping $u$ from $\R^l$ to a regular geodesic ball $B_R(P)\subset N$ is constant for $l=1,2,\cdots$ when $p\geq 2$. His idea actually works for the case $p\in (1,2)$. However, to overcome some additional difficulties that occurs when deriving the crucial Caccioppoli inequality, some delicate estimates of~\cite{hkw77} will be carefully and repeatedly applied.


Next we further estimate the gradient  of $C^1$-smooth weakly $p$-harmonic mappings. In~\cite[Theorem 2.2]{s84}, Schoen proved that there exists an $\varepsilon>0$ depending only on $n$, $g$ and $N$ such that if $u\colon B_r\to N$  is (minimizing) harmonic with $r^{2-n}\int_{B_r}|\nabla u|^2d\mu<\varepsilon$, then
$$\sup_{B_{r/2}}|\nabla u|^2\leq C\dashint_{B_r}|\nabla u|^2d\mu.$$
When $N$ is assumed to be non-positively curved, the gradient estimate as above still holds if we drop the smallness  assumption  on the normalized energy; see~\cite[Theorem 2.4]{gs92}.

This result was improved later by Korevaar and Schoen~\cite[Theorem 2.4.6]{ks93} in the following form: \emph{Let $\Omega$ be a smooth bounded domain of a Riemannian manifold $M$ and $N$ non-positively curved in the sense of Alexandrov. Suppose $u\colon \Omega\to N$ is minimizing harmonic. Then for any ball $B_R(o)$ with $B_{2R}(o)\subset\subset \Omega$, there exists a constant $C$ depending only on $n=\dim(M)$, $R$, the injectivity radius of $o$ and the $C^1$-norm of $g$ on $B_{2R}(o)$ such that}
$$\sup_{B_R(o)} |\nabla u|\leq C\dashint_{B_{2R}}|\nabla u|d\mu\leq C\Big(\dashint_{B_{2R}}|\nabla u|^2d\mu \Big)^{1/2}. $$
The dependence of the constant $C$ was further improved by Zhang, Zhong and Zhu in their very recent work~\cite{zzz17}\footnote{Indeed, the authors obtained quantitative gradient estimates for minimizing harmonic mappings from Riemannian manifolds with non-negative Ricci curvature into metric spaces with non-positive curvature in the sense of Alexandrov, which is much more general than the setting of Korevaar and Schoen.}.

Concerning the quantitative gradient estimate for stationary $p$-harmonic mappings, Duzaar and Fuchs proved in~\cite[Theorem 2.1]{df90} that, there exist $\varepsilon$ and $C$ depending only on $n$, $p$ and the curvature bound of $N$, such that if $u\colon B_r\to N$ is $C^1$-smooth weakly $p$-harmonic ($p\geq 2$) with the smallness condition $r^{p-n}\int_{B_r}|\nabla u|^pd\mu<\varepsilon$, then
$$\sup_{B_{r/2}}|\nabla u|^p\leq C \dashint_{B_r}|\nabla u|^pd\mu.$$

In this paper, we establish the quantitative gradient estimate for $C^1$-smooth weakly $p$-harmonic mappings when $M$ has non-negative Ricci curvature and $N$ is simply connected and has non-positive sectional curvature. As in the harmonic case~\cite{ks93}, the smallness condition for the normalized $p$-energy is unnecessary. Our third main result of this paper reads as follows.

\begin{theorem}\label{thm:gradient estimates}
	Assume that $M$ has non-negative Ricci curvature and $N$ is simply connected and has non-positive sectional curvature. Let $u\colon M\to N$ be a $C^1$-smooth weakly $p$-harmonic mapping. Then there exists a constant $C$, depending only on $n=\dim M$, such that for each  ball $B_r:=B_r(o)$ with $B_{2r}(o)\subset\subset M$, we have
	\begin{equation}\label{eq:gradient estimates}
	\sup_{B_{r}}|\nabla u|^{p-1}\leq C\dashint_{B_{2r}}|\nabla u|^{p-1}d\mu\leq C\Big(\dashint_{B_{2r}}|\nabla u|^{p}d\mu\Big)^{(p-1)/p}.
	\end{equation}
\end{theorem}

The proof of Theorem~\ref{thm:gradient estimates} follows closely the idea of Schoen and Yau~\cite{sy76}, which relies crucially on the Bochner-Weitzenb\"ock formula (due to Eells and Sampson~\cite{es64}). However, the degeneracy of $p$-harmonicity for $p\neq 2$ causes some extra technical difficulty. We tackle this difficulty by adapting some ideas from Duzaar and Fuchs~\cite[Proof of Theorem 2.1]{df90}, where the authors deal mainly with the case $p\geq 2$.

As an immediate consequence of Theorem~\ref{thm:gradient estimates}, we obtain the following Liouville-type theorem, which extends the classical result of Schoen and Yau \cite[Theorem 1.4]{sy76} for harmonic mappings to the setting of $p$-harmonic mappings.

\begin{corollary}\label{coro:Liouville from gradient estimate}
	Let $M=(M,g)$ be an $n$-dimensional complete non-compact Riemannian manifold with nonnegative Ricci curvature and $N$ a simply connected Riemannian manifold with non-positive sectional curvature. Then any $C^1$-smooth weakly $p$-harmonic mapping $u\colon M\to N$ with finite $(p-1)$ or $p$-energy must be constant.
\end{corollary}

\begin{proof}
	As $\mu(M)=\infty$ {(see e.g. \cite[Theorem 4.1 of Chapter I]{sy94})}, the result follows from~\eqref{eq:gradient estimates} by sending $r$ to infinite.
\end{proof}

Note that, under the assumption of Corollary~\ref{coro:Liouville from gradient estimate}, Nakauchi~\cite{n98} proved that any $C^1$-smooth weakly $p$-harmonic mapping $u\colon M\to N$ with finite $p$-energy must be constant for $p\geq 2$ via a different approach. Corollary~\ref{coro:Liouville from gradient estimate} extends this result to all $p\in (1,\infty)$.

\subsection{Structure of the paper}

This paper is structured as follows. The proofs of Theorem~\ref{thm:main theorem} and Theorem~\ref{thm:main theorem 2} are given in Section~\ref{sec:proof main thm} and Section~\ref{sec:proof main thm 2}, respectively. In Section~\ref{sec:gradient estimate}, we prove  Theorem~\ref{thm:gradient estimates}. The final section, Section~\ref{sec:concluding remarks}, contains some comments about our general method and possible extensions to mappings into more general metric spaces. We also include an appendix, establishing $W^{2,2}$ regularity estimates for weakly $p$-harmonic mappings in the case $1<p<2$, and as a byproduct, we extend the main results of Duzaar and Fuchs~\cite{df90} on gradient estimates and removable singularity of weakly $p$-harmonic mappings to the case $1<p<2$.

Our notation of various concepts is rather standard. Whenever we write $A(r)\lesssim B(r)$, it means that there exists a positive constant $C$, independent of $r$, such that $A(r)\leq CB(r)$.

\section{Proof of Theorem~\ref{thm:main theorem}}\label{sec:proof main thm}
In this section we assume that $N$ is simply connected and has non-positive sectional curvature and $\Omega\subset M$ is a domain.

Given a weakly $p$-harmonic mapping $u\colon M\to N$ whose image is contained in a compact subset of $N$, we will show in the following lemma that the composed function $d^2(u,Q)$ satisfies a weak differential inequality that relates the $p$-energy of $u$ and the gradient of $d^2(u,Q)$. In the (minimizing) harmonic case, this is due to Gromov and Schoen~\cite[Proposition 2.2]{gs92}.
\begin{lemma}\label{lemma:equation}
If $u\colon \Omega\to N$ is weakly $p$-harmonic with $u(\Omega)$ being contained in a compact subset of $N$, then for each $Q\in N$, the function $d^2(u,Q)$ satisfies the differential inequality
$$\int_{\Omega}|\nabla u|^{p-2}\big(2\eta|\nabla u|^2+\nabla \eta\cdot \nabla d^2(u,Q)\big)d\mu\leq 0 $$
for any $\eta\in C_0^\infty(\Omega)$.
\end{lemma}

\begin{proof}
Since $u\colon \Omega\to N$ is weakly $p$-harmonic, for any compactly supported vector field  $\psi \in W^{1,p}_0(M,\R^k)\cap L^\infty(M,\R^k)$  with $\psi (x)\in T_{u(x)}N\subset\R^k$ for a.e. $x\in \Om$,  we have
\begin{equation}\label{eq:for key lemma}
0=\frac{d}{dt}\Big|_{t=0}E_p(u_t)=\int_{\Omega}\langle |\nabla u|^{p-2}\nabla u, \nabla \psi\rangle d\mu,
\end{equation}
where $u_t(x)=\exp_{u(x)}\big(t\psi(x)\big)$.

Given $\eta\in C_0^\infty(\Omega)$ and $Q\in N$, we denote by $f:=d^2(x,Q)$. { Then $f\in C^2(N,\R)$. Since $u(\Omega)$ is contained in a compact subset of $N$, we may find a $C^2$-function $\bar{f}\colon \R^k\to \R$ with compact support in $\R^k$ such that $\bar{f}$ coincides with $f$ on $\overline{u(\Omega)}$. Set
$$\psi=\eta(x)(\nabla \bar{f})\circ u(x).$$
Then $\psi\in W^{1,p}_0(M,\R^k)\cap L^\infty(M,\R^k)$ is an admissible test vector field.}

Substitute {$\psi$} in~\eqref{eq:for key lemma} and we obtain
{
\begin{equation}
\begin{aligned}
	0&=\int_{\Omega}\big\langle |\nabla u|^{p-2}\nabla u, \nabla \big(\eta(x)\nabla \bar{f} \big)\big\rangle\\
	 &=\int_{\Omega}|\nabla u(x)|^{p-2}\big\langle \nabla u(x), \eta(x)\nabla_{\frac{\partial}{\partial x^\alpha}}(\nabla f)\otimes \nabla x^\alpha +\nabla \eta\otimes \nabla f\big\rangle d\mu\\
	 &=\int_{\Omega}|\nabla u|^{p-2}\Big( \eta(x)\nabla (\nabla f)(\nabla u,\nabla u)+\big\langle \nabla u(x), \nabla \eta\otimes \nabla f\big\rangle \Big) d\mu\\
	 &=\int_{\Omega}|\nabla u|^{p-2}\Big( \eta(x)\nabla (\nabla f)(\nabla u,\nabla u)+\big\langle \nabla\eta(x), \nabla (f\circ u)\big\rangle \Big) d\mu.
\end{aligned}
\label{eq:for key lemma 2}
\end{equation}
} On the other hand,  since $N$ is simply connected and non-positively curved, we have
$$\nabla (\nabla f)(\nabla u,\nabla u)\geq 2\|\nabla u\|^2;$$
see e.g.~\cite[Lemma 5.8.2]{j11}.
Inserting this into~\eqref{eq:for key lemma 2} yields
$$\int_{\Omega}|\nabla u|^{p-2}\big(2\eta|\nabla u|^2+\nabla \eta\cdot \nabla d^2(u,Q)\big)d\mu\leq 0. $$
The proof is complete.
\end{proof}

\begin{remark}\label{rmk:on key lemma}
Note that the assumption   $N$ being {simply connected and} non-positively curved is crucial in the above arguments as it implies that
$$\nabla(\nabla f)(v,v)\geq 2\|v\|^2$$
for the squared distance function $f=d^2(x,Q)$ (with any given $Q\in N$).
\end{remark}

We next derive the monotonicity formula for stationary $p$-harmonic mappings $u\colon M\to N$. Fix an arbitrary point $a\in M$ and set $E(r)=\int_{B_r(a)}|\nabla u|^pd\mu$. Note that $E'(r)=\int_{\partial B_r(a)}|\nabla u|^pd\Sigma$ for almost every $r$. When $M=\R^n$, the monotonicity formula (see~\cite[Lemma 4.1]{hl87}) for (minimizing) $p$-harmonic mappings  implies that for almost every $r\in (0,r_0)$,
$$\frac{d}{dr}\Big(r^{p-n}\int_{B_r(a)}|\nabla u|^pdx \Big)=pr^{p-n}\int_{\partial B_r(a)}|\nabla u|^{p-2}\Big|\frac{\partial u}{\partial r}\Big|^2 d\Sigma,$$
or we may equivalently formulate as
$$E'(r)=\frac{n-p}{r}E(r)+p\int_{\partial B_r(a)}|\nabla u|^{p-2}\Big|\frac{\partial u}{\partial r}\Big|^2 d\Sigma.$$

\begin{lemma}[Monotonicity formula]\label{lemma:monotone equation}
If $u\colon \Omega\to N$ is stationary $p$-harmonic, then for each $a\in {\Omega}$, there exists a radius $r_0>0$ such that for almost every $r\in (0,r_0)$, we have
	$$E'(r)=(1+O(r))\Big(\frac{n-p+O(r)}{r}E(r)+p\int_{\partial B_r(a)}|\nabla u|^{p-2}\Big|\frac{\partial u}{\partial r}\Big|^2 d\Sigma\Big).$$
\end{lemma}
\begin{proof}
	The proof is similar to the case $p=2$ from~\cite[Section 2, Page 192-193]{gs92}; see also~\cite[Lemma 3.1 and Lemma 3.2]{dm10}.
	
	Let $\eta$ be a smooth function with support in a small neighborhood of $a$. For $t$ small consider the diffeomorphism of $\Omega$ given in a normal coordinates by $F_t(x)=(1+t\eta(x))x$ in a neighborhood of 0 with $F_t=\text{id}$ outside this neighborhood. Consider the comparison mappings $u_t=u\circ F_t$. Then $u_t$ has the same trace and regularity as $u$. Since $u$ is stationary $p$-harmonic, $\frac{d}{dt}|_{t=0}E(u_t)=0$. Direct computation (see~\cite[Section 2, Page 192]{gs92}) gives
	$$0=\int_{\Omega}|\nabla u|^{p-2}\Big(|\nabla u|^2(p-n)\eta-|\nabla u|^2\sum_ix_i\frac{\partial \eta}{\partial x_i}+p\sum_{i,j,k}g^{ik}\frac{\partial \eta}{\partial x_i}x_j\frac{\partial u}{\partial x_j}\cdot\frac{\partial u}{\partial x_k}d\mu \Big)+A,$$
	where $A$ is the reminder term given by
	$$\int_{\Omega}|\nabla u|^{p-2}\Big(\eta\sum_{i,j,k}\frac{\partial g^{ij}}{\partial x_k}x_k\frac{\partial u}{\partial x_i}\cdot \frac{\partial u}{\partial x_j}\sqrt{g}+|\nabla u|^2\eta\sum_ix_i\frac{\partial \sqrt{g}}{\partial x_i} \Big)dx.$$
	Choosing $\eta$ to approximate the characteristic function of $B_r(a)$, we obtain
	$$0=rE'(r)-(n-p+O(r))E(r)-pr\int_{\partial B_r(a)}|\nabla u|^{p-2} \sum_{i,j,k}g^{ik}\frac{\partial \eta}{\partial x_i}x_j\frac{\partial u}{\partial x_j}\cdot\frac{\partial u}{\partial x_k}d\Sigma,$$
	where we have used the fact that the reminder term $|A|\leq crE(r)$ (because $|\frac{\partial g^{ij}}{\partial x_k}|,|\frac{\partial \sqrt{g}}{\partial x_i}|$ are bounded from above by some constant $c$). Since $g^{ik}\leq \delta^{ik}+cr$ when $r$ is sufficiently small, we get
	$$\sum_{i,j,k}g^{ik}\frac{\partial \eta}{\partial x_i}x_j\frac{\partial u}{\partial x_j}\cdot\frac{\partial u}{\partial x_k}\leq \Big|\frac{\partial u}{\partial r}\Big|^2+cr|\nabla u|^2,$$
	from which the claim follows.
\end{proof}


We would like to point out that the non-positive curvature assumption for $N$ was only used in Lemma~\ref{lemma:equation}, while, the conclusion of Lemma~\ref{lemma:monotone equation} remains valid for general Riemannian manifold $N$ (without any curvature restriction). With the aid of Lemma~\ref{lemma:equation} and Lemma~\ref{lemma:monotone equation}, we are able to derive the following important monotonicity inequality.

\begin{lemma}\label{lemma:key monotone inequality}
	There exist $r_1>0$ and $\gamma>0$ depending on $B_{r_0}(a)$, the Lipschitz bound and the ellipticity constant of $g$ such that
	$$r\mapsto \frac{E(r)}{r^{n-p+p\gamma}},\quad r\in (0,r_1)$$
	is non-decreasing.
\end{lemma}
\begin{proof}
Set $I_Q(r)=\int_{\partial B_{r}(a)}d^p(u,Q)d\Sigma$, where $r>0$ is small. Recall that  Poincar\'e's inequality {(see e.g. \cite[Lemma 2.1]{gx19})} for $B_{r_0}(a)$ implies that	
$$\inf_{Q\in N}I_Q(r)\leq Cr^p\int_{\partial B_r(a)}|\nabla u|^pd\Sigma,$$	
where the constant $C$ depends only on $B(a,r_0)$ and the ellipticity constant of $g$. We will fix $Q\in N$ such that the above Poincar\'e inequality holds for $u$.
	
We first consider the case $p\geq 2$. Choosing $\eta$ to approximate $\chi_{B_r(a)}$ in Lemma~\ref{lemma:equation} and then applying the H\"older's inequality and Poincar\'e inequality, we infer that
\begin{align*}
	E(r)^p&\lesssim \Big(\int_{\partial B_r(a)}|\nabla u|^{p-2}d(u,Q)\frac{\partial}{\partial  r}d(u,Q)d \Sigma\Big)^p\\
	&\leq I_Q(r)\Big(\int_{\partial B_r(a)}|\nabla u|^p d\Sigma  \Big)^{(p-2)/2}\Big(\int_{\partial B_r(a)}|\nabla u|^{p-2}\Big|\frac{\partial u}{\partial r}\Big|^2 d\Sigma \Big)^{p/2}\\
	&\lesssim r^{p/2}(rE'(r))^{p/2}\Big(\int_{\partial B_r(a)}|\nabla u|^{p-2}\Big|\frac{\partial u}{\partial r}\Big|^2 d\Sigma \Big)^{p/2}.
\end{align*}
 Set $A=\int_{\partial B_r(a)}|\nabla u|^{p-2}\Big|\frac{\partial u}{\partial r}\Big|^2 d\Sigma$. Lemma~\ref{lemma:monotone equation} and the above inequality imply that
  \begin{align*}
	E(r)^p&\lesssim
	(1+O(r))\Big((n-p+O(r))E(r)+prA\Big)^{p/2}(rA)^{p/2}\\
	&\lesssim \left((E(r))^{p/2}+(rA)^{p/2}\right)(rA)^{p/2}.
\end{align*}
  Note that the constant in the above estimate depends only on the constant from the Poincar\'e inequality and the ellipticity constant of $g$. Applying the Young's inequality $ab\le \epsilon a^2+C_{\epsilon}b^2$ (with $\varepsilon$ sufficiently small), we obtain from the previous inequality that
 $$E(r)\leq KrA$$ for some constant $K>0$ independent of $r$.

Now using Lemma~\ref{lemma:monotone equation} again, we have
\begin{align*}
rE'(r)&=(n-p+O(r))E(r)+(pr+O(r^2))A\\
&\geq (n-p+O(r))E(r)+\frac{p+O(r)}{K}E(r)\\
&=(n-p+\frac{p}{K}+O(r))E(r)\\
&\geq (n-p+p\gamma)E(r)
\end{align*} for some $\gamma>0$ when $r$ is sufficiently small.
This implies
$$\frac{d}{dr}\Big(\log\frac{E(r)}{r^{n-p+p\gamma}} \Big)\geq 0$$
and so the claim follows in this case.


Next we consider the case $1<p<2$. Similarly as in the previous case, we have
\[
\begin{aligned}E(r)^{p} & \lesssim\left(\int_{\pa B_{r}(a)}|\na u|^{p-2}d(u,Q)\left|\frac{\pa}{\pa r}d(u,Q)\right|d \Sigma\right)^{p}\\
 & \lesssim\left(\int_{\pa B_{r}(a)}d^{p}(u,Q)d \Sigma\right)\left(\int_{\pa B_{r}(a)}\left(|\na u|^{p-1-\ep}\left|\frac{\pa}{\pa r}d(u,Q)\right|^{\ep}\right)^{\frac{p}{p-1}}d \Sigma\right)^{p-1}\\
 & \lesssim I_{Q}(r)\left(\int_{\pa B_{r}(a)}|\na u|^{p-2}|\na u|^{2-p'\ep}\left|\frac{\pa}{\pa r}d(u,Q)\right|^{p'\ep}d \Sigma\right)^{p-1},
\end{aligned}
\]
{ where we used the estimate $\left|{\pa}_{r}d(u,Q)\right|\le C|\na u|$ (so that $\left|{\pa}_{r}d(u,Q)\right|/|\na u|\le C$) in the second line.}
Here, $p'=p/(p-1)$, $\ep>0$ is chosen such that $p-1-\ep>0$ and
$2-p'\ep>0$.  Applying H\"older's inequality and Poincar\'e's inequality,
we deduce
\[
\begin{aligned}E(r)^{p} & \lesssim I_{Q}(r)\left(\int_{\pa B_{r}(a)}|\na u|^{p}d \Sigma\right)^{p-1-\frac{p\ep}{2}}\left(\int_{\pa B_{r}(a)}|\na u|^{p-2}\left|\frac{\pa}{\pa r}d(u,Q)\right|^{2}d \Sigma\right)^{\frac{p\ep}{2}}\\
 & \lesssim r^{p}\left(E^{\prime}(r)\right)^{p-\frac{p\ep}{2}}A^{\frac{p\ep}{2}},
\end{aligned}
\]
which, according to Lemma~\ref{lemma:monotone equation}  and Young's inequality, implies that
\[
\begin{aligned}E(r) & \lesssim(rA)^{\ep/2}(rE^{\prime}(r))^{1-\ep/2}\\
 & \lesssim(rA)^{\ep/2}(E(r)+rA)^{1-\ep/2}\le\frac{1}{2}E(r)+KrA\\
\end{aligned}
\]
for some $K>0$ independent of $r$. The rest arguments are the same
as in the previous case. This completes the proof.
\end{proof}

\begin{proof}[Proof of Theorem~\ref{thm:main theorem}]
	By Lemma~\ref{lemma:key monotone inequality}, $u$ is locally H\"older continuous { by Morrey's Dirichlet growth theorem, see e.g. \cite[Chapter 3]{Giaquinta-Book}}. The local $C^{1,\alpha}$-regularity follows by the standard regularity theory of elliptic PDEs; see e.g.~\cite[Section 3]{hl87}. \end{proof}

\section{Proof of Theorem~\ref{thm:main theorem 2}}\label{sec:proof main thm 2}
\subsection{The $p$-minimizing tangent maps}

Following~\cite{su82} (for the case $p=2$) and~\cite{hl87}, we introduce the definition of minimizing tangent maps.

\begin{definition}\label{def:p-mim tangent map}
	 A mapping $v\in W^{1,p}_{loc}(\R^{l},N)$ is said to be a $p$-minimizing tangent map if $v\colon \R^l\to N$ is locally { minimizing} $p$-harmonic and is homogeneous of degree 0, that is, the radial derivative $\frac{\partial v}{\partial r}=0$ almost everywhere.
\end{definition}

Fix a $p$-harmonic mapping $u\colon \Omega\to N$ and an integer $l\in \{1,2,\cdots,n\}$. We consider the blow-up mappings $u_{x,r}(y):=u(x+ry)\colon \mathbb{B}\to N$, where $\mathbb{B}\subset \R^l$ is the unit open ball. By~\cite[Corollary 4.4]{hl87}, there exists a sequence $r_i\to 0$ such that $u_{x,r_i}$ converges strongly in $W^{1,p}(\mathbb{B},N)$ to a mapping $u_0\in W^{1,p}(\mathbb{B},N)$ which is homogeneous of degree 0. By homogeneity, we may then extend $u_0$ to all of $\R^l$ (and we still denote by $u_0$ the extended mapping) so that $u_0\colon \R^l\to N$ is a $p$-minimizing tangent map. We call such $u_0$ a $p$-minimizing tangent map of $u$.

Note that if $u(M)\subset B_R(P)$, then $u_{x,r_i}(\mathbb{B})\subset B_R(P)$ for each $i\in \mathbb{N}$. The strong convergence of $u_{x,r_i}$ to $u_0$ then implies that $u_0(\mathbb{B})\subset B_R(P)$. As $u_0$ is homogeneous of degree 0, $u_0(\R^l)\subset B_R(P)$ as well. Consequently, Theorem~\ref{thm:main theorem 2} follows immediately from~\cite[Theorem 4.5]{hl87} and the following Liouville's theorem for $p$-harmonic mappings from Euclidean space $\R^l$ into regular geodesic balls.
\begin{theorem}\label{thm:Liouville}
	There is no non-constant { minimizing} $p$-harmonic mapping $u\colon \mathbb{R}^l\to B_R(P)\subset N$ for each $l=1,2,\cdots$.
\end{theorem}

As commented earlier in the introduction, the case $p\geq 2$ has been proved by Fuchs~\cite{f88} and later the proof was extended to stationary $p$-harmonic mappings ($p\geq 2$) in~\cite{f89}, where the image is required to be contained in a smaller geodesic ball. We will give the proof of Theorem~\ref{thm:Liouville} in the next section, where we essentially extend the original arguments of Fuchs~\cite{f88} in combination with some arguments from~\cite{hkw77} to the case $p\in (1,2)$.

\subsection{Proof of Theorem~\ref{thm:Liouville}}

Fix a $p$-harmonic mapping $u\colon \R^l\to B_R(P)\subset N$. Let $h$ be the Riemannian metric on $N$. Before turning to the proof of Theorem~\ref{thm:Liouville}, we recall some elementary facts about $p$-harmonic mappings. In the following calculation, we will use the standard Einstein summation convention.

Let $v$ denote the representative of $u$ with respect to the normal coordinates centered in $B_R(P)$ and recall that
$$|\nabla v|=\Big(\delta^{\alpha\beta}h_{ij}(v)\frac{\partial v^i}{\partial x_\alpha}\frac{\partial v^j}{\partial x_\beta}\Big)^{{1}/{2}}.$$
Fix a ball $B\subset \R^l$. The Euler system for $v$ reads as
\begin{equation}\label{eq:Euler-Lang}
\int_{B}|\nabla v|^{p-2}\Big(h_{ij}(v)\frac{\partial v^i}{\partial x_\alpha}\frac{\partial \Phi^j}{\partial x_\beta}+\frac{1}{2}\big(\frac{\partial h_{ij}}{\partial x_k}\circ v\big)\frac{\partial v^i}{\partial x_\alpha}\frac{\partial v^j}{\partial x_\beta}\Phi^k \Big)\delta^{\alpha\beta} dx=0
\end{equation}for all bounded $\Phi\in W^{1,p}_0(B,\R^{{\rm dim} N})$. If we take $\Phi^k=h^{ki}(v)\Psi^i$, then
$$\frac{\partial \Phi^k}{\partial x_\beta}=h^{ki}(v)\frac{\partial \Psi^i}{\partial x_\beta}+\frac{\partial h^{ki}}{\partial v_m}\frac{\partial v^m}{\partial x_\beta}\Psi^i.$$
Plugging this into~\eqref{eq:Euler-Lang}, we finally arrive at
\begin{equation}\label{eq:Euler-Lang better version}
\int_{B}|\nabla v|^{p-2}\Big(\frac{\partial v^i}{\partial x_\alpha}\frac{\partial \Psi^j}{\partial x_\beta}\delta_{ij}-\Gamma_{ij}^l(v)\frac{\partial v^i}{\partial x_\alpha}\frac{\partial v^j}{\partial x_\beta}\Psi^l \Big)\delta^{\alpha\beta} dx=0,
\end{equation}
where $\Gamma_{ij}^l$ denotes the Christoffel symbols on the manifold $N$.

For each $x\in B$ with $r<d(x,\partial B)/2$, we define
\begin{equation*}
	\bar{V}:=\dashint_{B_{2r}(x)}Vdz\quad \text{and}\quad \bar{P}:=\big(\exp_P\big)^{-1}(\bar{V}),
\end{equation*}
where $V$ is the representation of $u$ with respect to the normal coordinates centered at $P$. { Note that  $\bar{P}$ is well-defined since $V$ is the representation of $u$ with respect to normal coordinates centered at $P$, $V$ takes its values into an Euclidean ball and we can see the exponential map as a map taking its value into this ball.}
Since $\bar{P}\in B_R(P)$, we may introduce another normal coordinates with center $\bar{P}$ and we denote by $v$ the representation of $u$ with respect to this normal coordinates.

Let $\eta\in C_0^\infty(B_{2r}(x))$ be a cut-off function which satisfies $\eta=1$ on $B_r(x)$, $0\leq \eta\leq 1$ in $B_{2r}(x)$ and $|\nabla \eta|\leq cr^{-1}$ for some constant $c=c(n)$. Set
$$\theta(v,\nabla v):=|\nabla v|^2-\Gamma_{ij}^l(v)\nabla v^i\cdot \nabla v^j v^l.$$
Note that
$$|v|=d(u,\bar{P})\leq d(u,P)+d(P,\bar{P})\leq 2R<\frac{\pi}{\sqrt{\kappa}}$$
and so by~\cite[Estimate (4.7)]{hkw77}
\begin{equation}\label{eq:lower bound for theta}
\theta(v,\nabla v)\geq a_{\kappa}(2R)h_{ij}(v)\nabla v^i\cdot \nabla v^j= a_{\kappa}(2R)|\nabla v|^2,
\end{equation}
where $a_{\kappa}>0$ is defined as in~\cite[Section 2]{hkw77}. Inserting~\eqref{eq:lower bound for theta} into~\eqref{eq:Euler-Lang better version} and taking $\Psi=\eta^pv$, we arrive at
\begin{equation}\label{eq:RHI 1}
\int_B \eta^pa_{\kappa}(2R)|\nabla v|^{p}dx\leq \sum_{\alpha,i}\int_B |\nabla v|^{p-2}\frac{\partial v^i}{\partial x_\alpha}v^i\frac{\partial \eta^p}{\partial x_\alpha}dx.
\end{equation}
Note that by~\cite[Lemma 1]{hkw77}, we have
\begin{equation}\label{eq:bounds from HKW lemma}
b_{\kappa}^2(|y|)|\xi|^2\leq h_{ij}(y)\xi^i\xi^k\leq b^2_\omega(|y|)|\xi|^2
\end{equation}
for all $\xi\in \R^k$, where $b_{\kappa}$ and $b_\omega$ are defined  as in~\cite[Lemma 1]{hkw77}. Applying~\eqref{eq:bounds from HKW lemma} with $y=v(x)$ and $\xi^i=\na_{\alpha} v^i$ for each fixed $\al$, we deduce
\begin{equation}\label{eq:bounds on A}
c_2|\nabla v(x)|^2\leq h_{ij}(v(x))\nabla v^i(x)\cdot \nabla v^j(x)\leq c_3|\nabla v(x)|^2
\end{equation}
for almost every $x\in B$. Applying~\eqref{eq:RHI 1},~\eqref{eq:bounds from HKW lemma} and $\varepsilon$-Young's inequality, we obtain
\begin{align*}
	a_{\kappa}(2R)\int_B \eta^p |\nabla v|^{p}dx & \leq \sum_{\alpha,i}\int_{B} |\nabla v|^{p-2}D_\alpha v^i v^iD_\alpha(\eta^p)dx\\
	&\le \varepsilon \int_B|\nabla v|^{p}\eta^{p}dx+c'(\varepsilon)\sum_i\int_B|\nabla \eta|^p|v^i|^pdx.
\end{align*}
Absorbing the first term into the left-hand side of the previous inequality, we obtain
\begin{equation}\label{eq:reverse Sobolev}
\int_B\eta^p|\nabla v|^{p}dx\leq c_4\int_B|\nabla \eta|^p|v|^pdx.
\end{equation}
Note that
$$\int_{B_r(x)}|\nabla v|^{p}dz=E_p(u|_{B_r(x)})\geq c_5\int_{B_r(x)}|\nabla V|^pdz,$$
where we have used the fact that an inequality of the form~\eqref{eq:bounds from HKW lemma} remains valid in normal coordinates centered at $P$. Observe that (see \cite[Page 11, footnote (1)]{hkw77})
$$|v(x)|=d(u(x),\bar{P})\leq b_\omega(2R)|V(x)-\bar{V}|.$$
Combining all these estimates, we arrive at the following Caccioppoli inequality for the coordinate representative of $u$
\begin{equation}\label{eq:key caccioppoli}
\int_{B_r(x)}|\nabla V|^pdz\leq c_6r^{-p}\int_{B_{2r}(x)}|V-\bar{V}|^pdz.
\end{equation}

\begin{remark}\label{rmk:on RHI}
1). The Caccioppoli inequality~\eqref{eq:key caccioppoli} was first obtained by Fuchs~\cite[Page 412]{f88}, where he assumed $p\geq 2$ and refers to the book of Giaquinta~\cite{Giaquinta-Book}. The proofs we are using here make use of some delicate estimates from~\cite[Proof of Theorem 3]{hkw77} { and is very similar to the proof given in \cite{f89} (notice that a regular geodesic ball always lies within normal range of all of its points and so one can check that the smaller radius requirement for regular geodesic ball in [13] is not needed in deriving the Cacciopoli inequality.)}. In particular, the Caccioppoli inequality~\eqref{eq:key caccioppoli} holds for \emph{weakly} $p$-harmonic mappings.

2). As a consequence of the Caccioppoli inequality~\eqref{eq:key caccioppoli} and~\cite[Lemma 5]{dm04}, we infer that if a \emph{weakly} $p$-harmonic mapping $u\colon B_{2r}\to N$ satisfies $u(B_{2r})\subset B_R(P)$ for a regular geodesic ball in $N$ and $E_p(u)\leq \varepsilon$ for some $\varepsilon$ depending only on $n$, $p$ and $N$, then $u\in C^{1,\alpha}(B_r,N)$ for some $\alpha$ depending only on $n$, $p$ and $N$.

3). Since~\eqref{eq:key caccioppoli} holds for all balls $B_{2r}(x)\subset\subset B$, we may apply the standard reverse H\"older inequality (see  Giaquinta~{\cite[Chapter V, Proposition 1.1]{Giaquinta-Book}}) to deduce that there is  $q>p$ such that $\na V\in L^q_{\rm loc}$. Moreover,
\begin{equation}\label{eq:reverse holder}
\Big(\dashint_{B_r(x)}|\nabla V|^qdz \Big)^{1/q}\leq c_7\Big(\dashint_{B_{2r}(x)}|\nabla V|^pdz \Big)^{1/p}.
\end{equation}
\end{remark}

Now we can prove Theorem~\ref{thm:Liouville}.

\begin{proof}[Proof of Theorem~\ref{thm:Liouville}]
	For each $k\in \mathbb{N}$, we set $u_k(x):=u(kx)$ and let $v$ and $v_k$ be the coordinate representation of $u$ and $u_k$ with respect to the normal coordinates centered at $P$.

 We first consider the case $p\leq l$. By the Caccioppoli inequality~\eqref{eq:key caccioppoli} we know
	\begin{equation*}
		\sup_k \|\nabla v_k\|_{L^p(B_t)}\leq c(t)
	\end{equation*}
	for all $t\in (0,\infty)$ with some constant $c(t)$ independent of $k$, where $B_t=B_t(0)\subset \R^l$. In particular, by the weak compactness of Sobolev spaces, we infer that there exists a $v_0\in W^{1,p}_{loc}(\R^l,N)$ such that $v_k$ converges to $v_0$ weakly in $W^{1,p}_{loc}(\R^l,N)$ and $v_k\to v$ pointwise almost everywhere. Thus $v_k$ converges to $v_0$ strongly in $W^{1,p}_{loc}(\R^l,N)$ as well (by~\cite[Lemma 2]{f88} or~\cite[Proposition 2]{Luckhaus88}\footnote{In fact, it was proved there that if a sequence of $p$-harmonic mappings $u_i$ converges weakly in $W^{1,p}$ to some mapping $u$, then the convergence is strong and $u$ is a $p$-harmonic mapping as well.}). Moreover, $v_0$ is homogenuous of degree 0, i.e., $\frac{\partial v_0}{\partial r}=0$ almost everywhere by the arguments of Fuchs~\cite[Page 413]{f88}, where only the monotonicity formula~\cite[Lemma 4.1]{hl87} is needed; see also~\cite[Proof of Proposition 2]{Luckhaus88}. Note that the strong convergence of $v_k$ to $v_0$ implies that $v_0$ also satisfies~\eqref{eq:Euler-Lang better version} and so we may select ${\Psi}(x):=\eta(|x|)v_0(x)$ with $\eta\in C_0^1\big((0,1)\big)$ and $\eta\geq 0$ to deduce that
	\begin{align*}
		0&=\int_{B_t}|\nabla v_0|^{p-2}\Big(|\nabla v_0|^2-\Gamma_{ij}^l\nabla v_0^i\cdot \nabla v_0^j v_0^l\Big)\eta dx\\
		&\geq a_{\kappa}(2R)\int_{B_t}|\nabla v_0|^{p-2}h_{ij}\nabla v_0^i\cdot \nabla v_0^j dx\\
		&\geq ca_{\kappa}(2R)\int_{B_t}|\nabla v_0|^{p}dx,
	\end{align*}
	where in the first equality we have used {the estimate \eqref{eq:lower bound for theta}} and the fact that
	$$\sum_\beta \frac{x_\beta}{|x|}D_\beta v_0^i=0 \quad \text{almost everywhere for all }i$$
	as $v_0$ is homogenuous of degree 0.
	Therefore, $\nabla v_0=0$ on $B_t$. Sending $t$ to infinite, we conclude that $\nabla v_0=0$ on $\R^l$. Now, using the monotonicity inequality again, we have for any $t\in (0,\infty)$
	\begin{align*}
		t^{p-l}\int_{B_t}|\nabla u|^pdx\leq (kt)^{p-l}\int_{B_{kt}}|\nabla u|^pdx=t^{p-l}\int_{B_t}|\nabla u_k|^p dx\to 0
	\end{align*}
	as $k\to \infty$. Thus $\nabla u=0$ on $B_t$ and hence also on $\R^l$.
	
	When $p>l$, the Liouville theorem follows directly from the Caccioppoli inequality~\eqref{eq:key caccioppoli}:
	\begin{align*}
		\int_{B_t}|\nabla u|^pdz\leq c_0t^{-p}\int_{B_{2t}}|u-\bar{u}|^pdz\leq R^pc_1t^{-p+l}\to 0
	\end{align*}
	as $t\to \infty$. Thus $\nabla u=0$ on $\R^l$. This completes our proof.	
\end{proof}

\begin{remark}\label{rmk:on compactness}

It would be interesting to know whether in the setting of Theorem \ref{thm:main theorem 2}, each \emph{weakly $p$-harmonic mapping} $u\colon \Omega\to N$  is continuous, as already conjectured by Fuchs \cite[page 131]{f89}. For $p=2$, this is the well-known result of Hildebrandt, Kaul and Widman~\cite{hkw77}, and for $p=n$, this follows immediately from the reverse H\"older inequality~\eqref{eq:reverse holder}.
\end{remark}

\section{Gradient estimates for stationary $p$-harmonic mappings}\label{sec:gradient
estimate}
In this section we assume $M$ has nonnegative Ricci curvature and $N$ has nonpositive sectional curvature. Recall that the Bochner-Weitzenb\"ock formula for $C^3$-smooth maps $u\colon M\to N$ reads as follows (see for instance~\cite[Lemma 1]{n98}):
\begin{align}\label{eq:Bochner formula}
	\frac{1}{2}\Delta\big(|du|^{2(p-1)}\big)=\langle \Delta\big(|du|^{p-2}du\big), |du|^{p-2}du\rangle+\Big|\nabla\big(|du|^{p-2}du \big) \Big|^2+|du|^{2(p-2)}R(du),
\end{align}
where the reminder term
\begin{equation}\label{eq:reminder term for Bochner}
	R(du)=\sum_i \langle \text{Ric}_M(du(e_i)), du(e_i)\rangle-\sum_{i,j}\langle R^N\big(du(e_i),du(e_j) \big)du(e_i), du(e_j)\rangle
\end{equation}
and $\Delta=-(dd^*+d^*d)$ is the Hodge-Laplace operator.
Note that
\begin{align*}
	\frac{1}{2}\Delta\big(|du|^{2(p-1)}\big)&=|du|^{p-1}\Delta(|du|^{p-1})+\Big|\nabla |du|^{p-1} \Big|^2\\
	&\leq |du|^{p-1}\Delta(|du|^{p-1})+\Big|\nabla\big(|du|^{p-2}du \big) \Big|^2.
\end{align*}
Thus,  it follows from~\eqref{eq:Bochner formula} that
\begin{equation*}
	|du|^{p-1}\Delta(|du|^{p-1})\geq \langle \Delta\big(|du|^{p-2}du\big), |du|^{p-2}du\rangle+|du|^{2(p-2)}R(du).
\end{equation*}
Equivalently, we have
\begin{equation}\label{eq:Bocher 3}
|du|\Delta(|du|^{p-1})\geq \langle \Delta\big(|du|^{p-2}du\big), du\rangle+|du|^{(p-2)}R(du).
\end{equation}
Note that if $\text{Ric}_M\geq 0$ and $R^N\leq 0$, then $|du|^{2(p-2)}R(du)\geq 0$.  So \eqref{eq:Bocher 3} reduces to
\begin{equation}\label{eq:Bocher 4}
|du|\Delta(|du|^{p-1})\geq \langle \Delta\big(|du|^{p-2}du\big), du\rangle.
\end{equation}

We now turn to the proof of Theorem~\ref{thm:gradient estimates}. Set $\Omega_+=\big\{x\in {M}: |\nabla u|>0\big\}$. We claim that for any non-negative $\eta\in C_0^1(\Omega_+)$, we  have
\begin{equation}\label{eq:test equation}
\int_{\Omega}\eta\langle \Delta\big(|du|^{p-2}du\big), du\rangle d\mu=0.
\end{equation}
Indeed, since $u$ is smooth $p$-harmonic in $\Omega_+$ and since $d^*\eta=0$, we have
\begin{align*}
\int_{\Omega}  \eta\langle \Delta\big(|du|^{p-2}du\big), du\rangle d\mu&=-\int_{\Omega} \eta\langle (dd^*+d^*d)\big(|du|^{p-2}du\big), du\rangle d\mu\\
&=-\int_{\Omega} \langle d^*\big(\eta d\big(|du|^{p-2}du\big)\big), du\rangle d\mu\\
&=-\int_{\Omega} \langle \eta d\big(|du|^{p-2}du\big), d(du)\rangle d\mu=0.
\end{align*}
Using a simple approximation argument, we may extend \eqref{eq:test equation} to all non-negative $\eta\in W^{1,2}_0(\Omega_+)$. Now we may divide $|du|$ on both side of \eqref{eq:Bocher 4} to obtain that
$$\Delta(|du|^{p-1})\geq |du|^{-1}\langle \Delta\big(|du|^{p-2}du\big), du\rangle.$$

We next observe that $|du|^{-1}\in W^{1,2}_{loc}(\Omega_{+})\cap L^{\infty}_{loc}(\Omega_+)$. Indeed, for $p\geq 2$, this follows directly from Duzaar and Fuchs~\cite[Page 391, -4 line]{df90}, and for $p\in (1,2)$, it follows from Proposition \ref{prop: second order derivatives} below.
Now for any non-negative $\eta\in C_0^1(\Omega_+)$, we have $|du|^{-1}\eta\in W^{1,2}_0(\Omega_+)$ and so it follows from \eqref{eq:test equation} that
\begin{equation}\label{eq:subharmonic}
\int_{\Omega}\Delta_g(|\nabla u|^{p-1})\eta dx\geq 0,
\end{equation}
where $\Delta_g$ is the standard Laplace-Beltrami operator on $M$. By~\cite[Lemma 2.4]{df90},~\eqref{eq:subharmonic} holds for all non-negative functions $\eta\in C_0^1(\Omega)$.

This implies that $|\nabla u|^{p-1}$ is a \emph{subharmonic function} on $M$ and so the standard theory for elliptic PDEs implies that there exists a positive constant $C$, depending only on $n$, such that
$$\sup_{B_{r}}|\nabla u|^{p-1}\leq C\dashint_{B_{2r}}|\nabla u|^{p-1}d\mu.$$
The desired inequality~\eqref{eq:gradient estimates} follows by applying H\"older's inequality. This completes the proof of Theorem~\ref{thm:gradient estimates}.


\section{Concluding remarks}\label{sec:concluding remarks}

In Theorem~\ref{thm:main theorem}, we have assumed that $u(M)$ is contained in a compact subset of $N$ and this assumption was used only in Lemma~\ref{lemma:equation}. This extra assumption can be dropped by a standard approximation argument if $W^{1,p}(M,N)\cap C^\infty(M,N)$ is dense in $W^{1,p}(M,N)$ (or actually even the under weaker density condition $W^{1,p}(M,N)\cap L^\infty(M,N)$ is dense in $W^{1,p}(M,N)$). This technical issue appears here because of the definition of Sobolev spaces and the choice of density for Sobolev mappings.

Let us recall the following definition of Sobolev spaces from~\cite{cs16}. A mapping $u\colon M\to N$ is said to be \emph{colocally weakly differentiable} if $u$ is measurable and $f\circ u$ is weakly differentiable for every smooth compactly supported function $f\in C^1_0(N,\R)$. For a colocally weakly differentiable mapping $u\colon M\to N$, a mapping $Du\colon TM\to TN$ is a colocal weak derivative of $u$ if $Du$ is a measurable bundle morphism that covers $u$ and
$$D(f\circ u)=Df\circ Du$$
holds almost everywhere in $M$ for every $f\in C^1_0(N,\R)$. A mapping $u\colon M\to N$ belongs to the Sobolev space $W^{1,p}_{cs}(M,N)$ if $u\in L^p(M,N)$ is colocally weakly differentiable and the norm of the colocal weak differential $|Du|_{g_M^*\otimes g_N}\in L^p(M)$.

In many aspects, colocal weak derivatives behave as nicely as weak derivatives of mappings between Euclidean spaces. In particular, for a $C^1$-smooth mapping $u\colon M\to N$, the colocal weak derivative coincides with the classical weak derivative almost everywhere. Moreover, one can show that the Sobolev space $W^{1,p}_{cs}(M,N)$ is equivalent to the Sobolev space $W^{1,p}(M,N)$ defined as in Section~\ref{subsec:main results}; see~\cite[Proposition 2.6]{cs16}. Thus we can develop a theory for $p$-harmonic mappings based on the colocal weak derivative as $Du$ (and thus $|Du|$) is well-defined. In this case, one would expect Lemma~\ref{lemma:equation} holds with $Du$ in place of $\nabla u$ as $f=d^2(x,Q)$ is Lipschitz on $N$ and $f\circ u$ would be weakly differentiable; see~\cite[Proposition 2.1]{cs16}. Moreover, Lemma~\ref{lemma:monotone equation} and Lemma~\ref{lemma:key monotone inequality} remain valid as only nice computation law for ``derivatives" are needed.

For simplicity of our exposition, we did not consider this issue in the current paper, but we will present all the details in a forth-coming work, together with extensions to Finsler/SubRiemannian manifolds.

\subsection*{Acknowledgements}
C.-Y. Guo would like to thank Prof.~M. Fuchs and Prof.~H.-C. Zhang for their interest on this work and for their valuable communications. He also wants to thank the excellent event \emph{``Mathematics, Physics, and their Interaction. Conference in Honour of Demetrios Christodoulou's 65th Birthday"} held at ETH Zurich in July 2017, where part of this work was done. Both authors are grateful to Prof.~L.-Q. Yang for her insightful comments on the Bochner-Weitzenb\"ock formula, and to Prof.~G. Veronelli for his helpful comments on Lemma~\ref{lemma:equation}. Both authors would also like to thank anonymous referees for many useful suggestions and comments which improve this work a lot.

\appendix

\section{$W^{2,2}$ regularity  and removable singularities of $p$-harmonic mappings:  $1<p<2$}

To derive the boundedness of gradients of $C^1$-smooth weakly $p$-harmonic mappings (which is needed in Section~\ref{sec:gradient estimate}), we need a $W^{2,2}$ regularity estimate. In the case $p\ge2$, this type of result has been established by Duzaar and Fuchs \cite{df90}. We believe the corresponding results, for the case $1<p<2$, are also well-known among specialists in the field. But, since we do not find a precise reference for such a result, we decide to include
a sketch of proof below. We will apply the method of Acerbi and Fusco \cite{af89},
where, among other results, $W^{2,2}$ regularity estimates for $p$-harmonic
mappings ($1<p<2$) between Euclidean spaces were established.

From now on, we stick to the assumption $1<p<2$. Let $\Om\subset\R^{n}$ be an
open set and $N$ a smooth Riemannian manifold that is isometrically
embedded in some Euclidean space $\R^{k}$ with $k\in \mathbb{N}$. Let
$u\colon \Om\to N$ be a $C^1$-smooth \emph{weakly $p$-harmonic mapping}, that is,
$u$ satisfies the $p$-Laplace equation
\begin{eqnarray}
\int_{\Om}|\na u|^{p-2}\na_{\al}u\cdot\na_{\al}\var+\int_{\Om}|\na u|^{p-2}A(u)(\na_{\al}u,\na_{\al}u)\cdot\var=0, &  & \forall\:\var\in C_{0}^{1}(\Om,\R^{k}),\label{eq: p-harm equ.}
\end{eqnarray}
where $A(q)(\cdot,\cdot):T_{q}N\times T_{q}N\to\left(T_{q}N\right)^{\bot}$
is the second fundamental form of $N$ at $q\in N$. Note that Einstein summation
convention over $\al$ from $1$ to $n$ is applied above. We further assume
that $N$ satisfies the curvature assumptions (1.4) and (1.5) of $N$
prescribed in Duzaar and Fuchs \cite{df90}.

\subsection*{$W^{2,2}$ regularity of $p$-harmonic mappings for $1<p<2$}
In this section, we will establish an interior $W^{2,2}$ regularity estimate
of $u$, and then in the next section, extend the main result of Duzaar and Fuchs \cite{df90}
with a sketch of proof.

We will use the following elementary inequality, which is a consequence of Lemma 2.2 of Acerbi and Fusco \cite{af89}:
for any $l\ge1$, there exists a constant $c=c(l,p)>0$ such that for any $a,b\in\R^{l}$,
\begin{equation}
c\frac{|a-b|^{2}}{\left(|a|^{2}+|b|^{2}\right)^{\frac{2-p}{2}}}\ge\langle|a|^{p-2}a-|b|^{p-2}b,a-b\rangle\ge(p-1)\frac{|a-b|^{2}}{\left(|a|^{2}+|b|^{2}\right)^{\frac{2-p}{2}}}.\label{eq: AF inequality}
\end{equation}

The main result of this section reads as follows.

\begin{proposition}\label{prop: second order derivatives}
Each mapping $u\in W^{1,p}(\Om,N)\cap C^{1}(\Om,N)$ that satisfies \eqref{eq: p-harm equ.}  belongs to  $W_{\loc}^{2,2}(\Om_{+},N)$, where $\Om_{+}=\{x\in\Om:|\na u(x)|>0\}$.
Moreover, there exists a constant $C>0$ depending only on $n,k,p$ and the
curvature assumptions on $N$, such that for any $B_{r}\subset\subset\Om$,
it holds
\[
\int_{B_{r/2}}|\na^{2}u|^{2}\le C\left(r^{-2}+M_{r}^{2}\right)M_{r}^{2-p}\int_{B_{r}}|\na u|^{p},
\]
where $M_{r}=\sup_{B_{r}}|\na u|$.
\end{proposition}
\begin{proof}
Let $B_{r}\subset\subset\Om_{+}$ and $h>0$ be sufficiently small.
For fixed $1\le\be\le n$, we denote
\[
\De_{h}f(x)=\frac{1}{h}(f(x+he_{\be})-f(x))
\]
and set
\[
V=|\na u|^{\frac{p-2}{2}}\na u.
\]

By (\ref{eq: p-harm equ.}), we have
\begin{equation}
\int_{\Om}\De_{h}\left(|\na u|^{p-2}\na_{\al}u\right)\cdot\na_{\al}\var=-\int_{\Om}\De_{h}\left(A(u)(V,V)\right)\cdot\var\label{eq: difference quotient equ.}
\end{equation}
for any $\var\in C_{0}^{1}(\Om,\R^{k})$. It is easy to see that the above
equation holds for all $\var\in W_{0}^{1,p}\cap L^{\wq}(\Om_{+},\R^{k})$ as well.
Substitute $\var=\eta^{2}\De_{h}u$ into { the left hand side of} (\ref{eq: difference quotient equ.})
for $\eta\in C_{0}^{2}(\Om_{+})$ and we obtain
\[
\begin{aligned}\int\De_{h}\left(|\na u|^{p-2}\na_{\al}u\right)\cdot\na_{\al}\var & =\int\eta^{2}\De_{h}\left(|\na u|^{p-2}\na_{\al}u\right)\cdot\De_{h}\na_{\al}u\\
 & \quad\qquad+\int2\eta\De_{h}\left(|\na u|^{p-2}\na_{\al}u\right)\cdot\De_{h}u\na_{\al}\eta.
\end{aligned}
\]
Applying (\ref{eq: AF inequality}), we deduce
\[
\int\eta^{2}\De_{h}\left(|\na u|^{p-2}\na_{\al}u\right)\cdot\De_{h}\na_{\al}u\ge c\int\eta^{2}\left(|\na u(x)|+|\na u(x+he_{\be})|\right)^{p-2}|\De_{h}\na u|^{2}
\]
and
\[
\begin{aligned} & \left|\int2\eta\De_{h}\left(|\na u|^{p-2}\na_{\al}u\right)\cdot\De_{h}u\na_{\al}\eta\right|\\
 & \qquad\le c^{\prime}\int\left(|\na u(x)|+|\na u(x+he_{\be})|\right)^{p-2}|\De_{h}\na u||\De_{h}u|\eta|\na\eta|
\end{aligned}
\]
for some constants $c,c^{\prime}>0$ depending only on $p$. Combining it with
Young's inequality gives us
\begin{equation}
\begin{aligned}\int_{\Om}\De_{h}\left(|\na u|^{p-2}\na_{\al}u\right)\cdot\na_{\al}\var & \ge c_{1}\int_{\Om}\eta^{2}\left(|\na u(x)|+|\na u(x+he_{\be})|\right)^{p-2}|\De_{h}\na u|^{2}\\
 & \quad-c_{2}\int_{\Om}\left(|\na u(x)|+|\na u(x+he_{\be})|\right)^{p-2}|\De_{h}u|^{2}|\na\eta|^{2}
\end{aligned}
\label{eq: LHS}
\end{equation}
for some constants $c_{1},c_{2}>0$ depending only on $p$ and $k$.

On the other hand, by estimate (2.6) of Duzaar and Fuchs \cite{df90},
we have
\begin{equation}
\begin{aligned}\left|\int_{\Om}\De_{h}\left(A(u)(V,V)\right)\cdot\eta^{2}\De_{h}u\right| & \le\frac{c_{1}}{2}\int_{\Om}\eta^{2}\left(|\na u(x)|+|\na u(x+he_{\be})|\right)^{p-2}|\De_{h}\na u|^{2}\\
 & \quad+c\int_{\Om}\left(|\na u(x)|+|\na u(x+he_{\be})|\right)^{p}|\De_{h}u|^{2}\eta^{2}
\end{aligned}
\label{eq: RHS}
\end{equation}
for some $c>0$ depending only on $n,p,k$ and the curvature assumptions
on $N$. Hence, combining (\ref{eq: difference quotient equ.}), (\ref{eq: LHS})
and (\ref{eq: RHS}) yields
\[
\begin{aligned} & \int_{\Om}\eta^{2}\left(|\na u(x)|+|\na u(x+he_{\be})|\right)^{p-2}|\De_{h}\na u|^{2}\\
 & \quad\le c\int_{\Om}\left(|\na u(x)|+|\na u(x+he_{\be})|\right)^{p-2}|\De_{h}u|^{2}|\na\eta|^{2}\\
 & \quad+c\int_{\Om}\left(|\na u(x)|+|\na u(x+he_{\be})|\right)^{p}|\De_{h}u|^{2}\eta^{2}.
\end{aligned}
\]

Now choose $\eta\in C_{0}^{\wq}(B_{3r/4})$ such that $\eta\equiv1$
on $B_{r/2}$ and $|\na\eta|\le8/r$ and $|\na^{2}\eta|\le8/r^{2}$.
Recall that $u\in C^{1}(\Om,N)$ and we obtain from the above that
\[
\begin{aligned}\int_{B_{r/2}}|\De_{h}\na u|^{2} & \le cM_{r}^{2-p}r^{-2}\int_{\Om}\left(|\na u(x)|+|\na u(x+he_{\be})|\right)^{p-2}|\De_{h}u|^{2}\\
 & \quad+cM_{r}^{2-p}M_{r}^{2}\int_{B_{3r/4}}\left(|\na u(x)|+|\na u(x+he_{\be})|\right)^{p}.
\end{aligned}
\]
Letting $h\to0$ yields $\na^{2}u\in L^{2}(B_{r/2})$ and the desired
estimate. The proof is complete.
\end{proof}

\subsection*{Removable singularities of $p$-harmonic mappings for $1<p<2$}
We next point out that the main result of Duzaar and
Fuchs \cite[Theorem, page 386]{df90} holds for the case $1<p<2$ as well.

\begin{theorem}\label{thm: Duzaar-Fuchs theorem} Let $n\ge2$ and
$1<p<2$. Suppose $u\in C^{1}(B_{1}\backslash\{0\},N)\cap W^{1,p}(B_{1},N)$
is a weakly $p$-harmonic mapping (that is, $u$ solves (\ref{eq: p-harm equ.})).
Then, there exists a constant $\ep_{0}>0$ depending only on $n,k,p$ and the
geometry of $N$, such that if the $p$-energy of $u$ satisfies
\[
E_{p}(u)\equiv\int_{B_{1}}|\na u|^{p}\le\ep_{0},
\]
then $u\in C^{1,\ga}(B_{1},N)$ for some $\ga\in(0,1)$. Moreover, the H\"older
exponent $\ga$ depends only on $n,k,p$ and the geometry of $N$.
\end{theorem}

The proof of Theorem \ref{thm: Duzaar-Fuchs theorem} follows closely
the arguments of Duzaar and Fuchs \cite{df90} with  minor modifications.
 { Below}, we list the main ingredients and point out the corresponding modifications.

The first ingredient is the following quantitative gradient estimates for $p$-harmonic mappings, which extends Theorem 2.1 of Duzaar and Fuchs \cite{df90} to the case $1<p<2$.

\begin{proposition} \label{prop: gredient estimates under small energy assumption}
Let $1<p<2$. Assume that $u\in C^{1}(B_{r},N)$ is a weakly $p$-harmonic
mapping. Then, there exist constants $\ep_{1},C_{1}>0$ depending only on $n,k,p$
and the geometry of $N$, such that if $r^{p-n}\int_{B_{r}}|\na u|^{p}\le\ep_{1},$
then
\[
\sup_{B_{r/2}}|\na u|^{p}\le C_{1}\dashint_{B_{r}}|\na u|^{p}.
\]
\end{proposition}

In the case $p\ge2$, the above result is Theorem 2.1 of Duzaar
and Fuchs \cite{df90}. The key ingredient in the proof of Theorem 2.1
is to derive $W^{2,2}$ type regularity estimates for $p$-harmonic mappings; see Lemma 2.2 of Duzaar and Fuchs \cite{df90}. In our case, one can easily check that, with the $W^{2,2}$ regularity estimates (Proposition~\ref{prop: second order derivatives}) at hand, the rest arguments of Duzaar and Fuchs \cite{df90} can be applied without changes.

The second ingredient is the following proposition, which extend Proposition
3.1 of Duzaar and Fuchs \cite{df90} to the case $1<p<2$.

\begin{proposition} \label{prop: DF proposition 3.1} There exist constants
$\ep_{0}>0$, $\si\in(0,1)$, depending only on $n,k,p$ and the curvature
assumptions of $N$, such that for any weakly $p$-harmonic mapping
\[
u\in C^{1}(B_{1}\backslash\{0\},N)\cap W^{1,p}(B_{1},\R^{k})
\]
with $\int_{B_{1}}|\na u|^{p}\le\ep_{0}$, it holds
\[
\si^{p-n}E(\si)\le\frac{1}{2}E(1),
\]
where we used the notation $E(r)=\int_{B_{r}}|\na u|^{p}$.
\end{proposition}

 To establish this result for $1<p<2$, we only need to show that similar estimates as equations (3.5) and (3.10) of Duzaar and Fuchs \cite{df90} holds for the case $1<p<2$ as well.

Let $\{v_{i}\}$ be defined as that of \cite[Page 397]{df90}. Then,
by the same arguments as that of \cite{df90}, we have
\[
\int_{B_{{1}/{2}}\backslash B_{r}}\left(|\na v_{i}|^{p-2}\na_{\al}v_{i}-|\na v_{j}|^{p-2}\na_{\al}v_{j}\right)\cdot\left(\na_{\al}v_{i}-\na_{\al}v_{j}\right)\eta^{p}\to0
\]
 as $i,j\to\wq$. Then, (\ref{eq: AF inequality}) implies\footnote{Note that there is a typos in (3.5) of Duzaar and Fuchs \cite{df90}:
the first symbol $\wq$ in (3.5) should be $0$.}
\begin{eqnarray*}
\int_{B_{1/2}\backslash B_{r}}\left(|\na v_{j}|^{2}+|\na v_{i}|^{2}\right)^{{\frac{p-2}{2}}}|\na v_{i}-\na v_{j}|^{2}\eta^{p}\to0, &  & \text{as }i,j\to\wq,
\end{eqnarray*}
from which we deduce that $v_{i}\to v_{\wq}$ strongly in $W^{1,p}(B_{1/2}\backslash B_{r})$
for some $v_{\wq}$ in $W^{1,p}(B_{1/2}\backslash B_{r})$, in view of the following H\"older inequality
\[
\int|\na v_{i}-\na v_{j}|^{p}\eta^{p}\le\left(\int\frac{|\na v_{i}-\na v_{j}|^{2}}{\left(|\na v_{j}|^{2}+|\na v_{i}|^{2}\right)^{\frac{2-p}{2}}}\eta^{p}\right)^{\frac{p}{2}}\left(\int\left(|\na v_{j}|^{2}+|\na v_{i}|^{2}\right)^{\frac{p}{2}}\eta^{p}\right)^{\frac{2-p}{2}}
\]
and the fact that $\{v_{j}\}$ is uniformly bounded in $W^{1,p}(B_{1},\R^{k})$.
Hence (3.5) of Duzaar and Fuchs \cite{df90} holds for $1<p<2$ as
well.

As to the estimate (3.10) of Duzaar and Fuchs \cite{df90}, it has
been established in the case $1<p<2$ by Acerbi and Fusco \cite[Proposition 2.7]{af89}.

The rest { of the} arguments of Duzaar and Fuchs \cite{df90} remains valid
for $1<p<2$, and so Proposition \ref{prop: DF proposition 3.1} holds.

With Propositions \ref{prop: gredient estimates under small energy assumption}
and \ref{prop: DF proposition 3.1} at hand, Theorem \ref{thm: Duzaar-Fuchs theorem}
follows by the same arguments as that of Duzaar and Fuchs \cite{df90} and so we omit the details.

%
%
%

\end{document}